\documentclass{article}

\usepackage{multirow}
\usepackage{booktabs}
\usepackage{enumitem}
\usepackage{algorithm}
\usepackage{algorithmicx}
\usepackage{algpseudocode}
\usepackage{url}

\usepackage{graphicx}

\usepackage{amsmath, amssymb, amsthm}

\usepackage{cite}
\usepackage{amsmath,amssymb,amsthm,mathrsfs}
\newtheorem{theorem}{Theorem}[section]

\newtheorem{lemma}[theorem]{Lemma}
\newtheorem{proposition}[theorem]{Proposition}
\theoremstyle{definition}
\newtheorem{definition}[theorem]{Definition}

\newtheorem{assum}{Assumption}[section]

\DeclareMathOperator*{\grad}{grad}
\DeclareMathOperator*{\tr}{tr}
\newcommand{\real}{\mathbb{R}}
\newcommand{\nat}{\mathbb{N}}
\newcommand{\spd}{\mathcal{S}}
\newcommand{\epc}{\mathbb{E}}

\newcommand{\Exp}{\mathrm{Exp}}
\newcommand{\abs}[1]{\left|{#1}\right|}
\newcommand{\norm}[1]{\left\lVert{#1}\right\rVert}
\newcommand{\ip}[2]{\left\langle{#1,#2}\right\rangle}
\newcommand{\relmiddle}[1]{\mathrel{}\middle#1\mathrel{}}

\numberwithin{equation}{section}

\begin{document}
\makeatletter

\begin{center}
\large{\bf Convergence of Riemannian Stochastic Gradient Descent on Hadamard Manifold}
\end{center}\vspace{5mm}
\begin{center}

\textsc{HIROYUKI SAKAI, HIDEAKI IIDUKA}\end{center}
\vspace{2mm}

\footnotesize{
\noindent\begin{minipage}{14cm}
{\bf Abstract:}
Novel convergence analyses are presented of Riemannian stochastic gradient descent (RSGD) on a Hadamard manifold.
RSGD is the most basic Riemannian stochastic optimization algorithm and is used in many applications in the field of machine learning.
The analyses incorporate the concept of mini-batch learning used in deep learning and overcome several problems in previous analyses.
Four types of convergence analysis are described for both constant and decreasing step sizes.
The number of steps needed for RSGD convergence is shown to be a convex monotone decreasing function of the batch size.
Application of RSGD with several batch sizes to a Riemannian stochastic optimization problem on a symmetric positive definite manifold theoretically shows that increasing the batch size improves RSGD performance.
Numerical evaluation of the relationship between batch size and RSGD performance provides evidence supporting the theoretical results.
\end{minipage}
 \\[5mm]


\hbox to14cm{\hrulefill}\par


\section{Introduction}
Riemannian optimization has attracted a great deal of attention \cite{absil2008optimization, gao2020learning, sakai2022riemannian} in the field of machine learning.
In this paper, we consider a Hadamard manifold, which is a complete Riemannian manifold of which the sectional curvatures are less than or equal to zero.
From the viewpoints of application and convergence analysis, optimization problems on a Hadamard manifold are very important.
This is because, from the Cartan-Hadamard theorem, there exists an inverse map of the exponential map.
Hyperbolic spaces (represented, for example, by the Poincar{\'e} ball model and Poincar{\'e} half-plane model)
and symmetric positive definite (SPD) manifolds are examples of Hadamard manifolds with many applications in the field of machine learning.

Optimization problems on SPD manifolds have a variety of applications in the computer vision and machine learning fields.
In particular, visual representations often rely on SPD manifolds, such as the kernel matrix, the covariance descriptor \cite{tuzel2006region}, and the diffusion tensor image \cite{fletcher2007riemannian}.
Optimization problems on SPD manifolds are especially important in the medical imaging field \cite{cherian2017learning, fletcher2007riemannian}.
Furthermore, optimization problems in hyperbolic spaces are important in natural language processing.
Nickel and Kiela \cite{nickel2017poincare} proposed using Poincar{\'e} embedding, an example of a Riemannian optimization problem in hyperbolic space.
More specifically, they proposed embedding hierarchical representations of symbolic data (e.g., text, graph data) into the Poincar{\'e} ball model or Poincar{\'e} half-plane model of hyperbolic space. 

Riemannian stochastic optimization algorithms are used for solving these optimization problems on a Riemannian manifold.
Various such algorithms have been developed by extending gradient-based optimization algorithms in Euclidean space.
Bonnabel proposed using Riemannian stochastic gradient descent (RSGD), which is the most basic Riemannian stochastic optimization algorithm \cite{bonnabel2013stochastic}.
Sato, Kasai, and Mishra proposed using the Riemannian stochastic variance reduced gradient (RSVRG) algorithm with a retraction and vector transport \cite{sato2019riemannian}.
Moreover, they gave a convergence analysis of RSVRG under certain reasonable assumptions.
In general, the RSVRG algorithm converges to an optimal solution faster than RSGD;
however, the full gradient needs to be calculated every few steps with RSVRG.

Adaptive optimization algorithms such as AdaGrad \cite{duchi2011adaptive}, Adadelta \cite{zeiler2012adadelta}, Adam \cite{kingma2015adam},
and AMSGrad \cite{reddi2018convergence} are widely used for training deep neural networks in Euclidean space.
However, they cannot be easily extended to general Riemannian manifolds due to the absence of a canonical coordinate system.
Special measures must therefore be considered when extending them to Riemannian manifolds.
Kasai, Jawanpuria, and Mishra proposed generalizing adaptive stochastic gradient algorithms to Riemannian matrix manifolds by adapting the row and column subspaces of the gradients \cite{kasai2019riemannian}.
B{\'e}cigneul and Ganea proposed using the Riemannian AMSGrad (RAMSGrad) algorithm \cite{becigneul2018riemannian};
however, RAMSGrad is defined only on the \emph{product} of Riemannian manifolds by regarding each component of the product Riemannian manifold as a coordinate component in Euclidean space.

Bonnabel presented two types of RSGD convergence analysis on a Hadamard manifold \cite{bonnabel2013stochastic},
but both of them are based on unrealistic assumptions, and use only diminishing step sizes.
The first type is based on the assumption that the sequence generated by RSGD is contained in a compact set of a Hadamard manifold $M$.
Since it is difficult to predict the complete sequence, this assumption should be removed.
The second type is based an unrealistic assumption regarding the selection of step size.
Specifically, a function ($v:M\to\real$) determined by a Riemannian optimization problem must be computed, and a diminishing step size ($\alpha_k$ divided by $v(x_k)$, where $x_k$ is a $k$-th approximation defined by RSGD) must be used.
This is not a realistic assumption because the step size is determined by the Riemannian optimization problem to be solved and must be adapted \emph{manually}.

We have improved RSGD convergence analysis on a Hadamard manifold in accordance with the above discussion and present four types of convergence analysis for both constant and diminishing step sizes (see Section \ref{sec:analysis}).
First, we consider the case in which an objective function $f:M\to\real$ is $L$-smooth (Definition \ref{dfn:Lip}).
Theorems \ref{thm:Lip-cnv-cns} and \ref{thm:Lip-cnv-dim} are for convergence analyses with $L$-smooth assumption for constant and diminishing step sizes, respectively.
Since calculating the constant $L$ in the definition of $L$-smooth is often difficult, we also present convergence analyses for the function $f:M \to \real$ \emph{not} $L$-smooth.
Theorems \ref{thm:nLip-cnv-cns} and \ref{thm:nLip-cnv-dim} support convergence analyses without $L$-smooth assumption for constant and diminishing step sizes, respectively.
We summarize the existing and proposed analyses in Table \ref{tab:improvements} 

Moreover, we show that the number of steps $K$ needed for an $\varepsilon$-approximation of RSGD is convex and monotone decreasing with respect to batch size $b$.
We also show that SFO complexity (defined as $Kb$) is convex with respect to $b$ and that there exists a critical batch size such that SFO complexity is minimized (see Section \ref{sec:critical}).

This paper is organized as follows.
Section \ref{sec:preliminaries} reviews the fundamentals of Riemannian geometry and Riemannian optimization.
Section \ref{sec:analysis} presents four novel convergence analyses of RSGD on a Hadamard manifold.
Section \ref{sec:experiments} experimentally evaluate RSGD performance by solving the Riemannian centroid problem on an SPD manifold for several batch sizes and evaluate the relationship between the number of steps $K$ and batch size $b$.
Section \ref{sec:conclusion} concludes the paper.

\begin{table}[htbp]
\centering
\caption{Assumptions used in existing convergence analyses given by \cite{bonnabel2013stochastic} and our novel convergence analyses
(Theorems \ref{thm:Lip-cnv-cns}--\ref{thm:nLip-cnv-dim}). In all cases, $M$ is a Hadamard manifold, and $f: M \to \real$ is a smooth function.\label{tab:improvements}}
\begin{tabular}{cccc}
\bottomrule
\hline
\multirow{2}{*}{Theorem} &\multicolumn{3}{c}{Assumptions} \\
\cmidrule(lr){2-4}
& Step size $\alpha_k$ & Function $f$ & Sequence $x_k$ \\ \hline
 \toprule
\cite{bonnabel2013stochastic} & \multirow{2}{*}{Diminishing} & \multirow{2}{*}{--} & $(x_k)_{k=0}^{\infty} \subset C$ \\
Theorem 2& & & $C$: compact set \\ \hline
\cite{bonnabel2013stochastic} & Determined by & \multirow{2}{*}{--} & \multirow{2}{*}{--} \\
Theorem 3 & problem & \\ \hline
\multirow{2}{*}{Theorem \ref{thm:Lip-cnv-cns}} & Constant & \multirow{2}{*}{$L$-smooth} & \multirow{2}{*}{--} \\
& depending on $L$ & & \\ \hline
\multirow{2}{*}{Theorem \ref{thm:Lip-cnv-dim}} & Diminishing & \multirow{2}{*}{$L$-smooth} & \multirow{2}{*}{--} \\
& not depending on $L$ & \\ \hline
\multirow{2}{*}{Theorem \ref{thm:nLip-cnv-cns}} & \multirow{2}{*}{Constant} & \multirow{2}{*}{--} & $(x_k)_{k=0}^{\infty} \subset C$ \\
& & & $C$: bounded set \\ \hline
\multirow{2}{*}{Theorem \ref{thm:nLip-cnv-dim}} & \multirow{2}{*}{Diminishing} & \multirow{2}{*}{--} & $(x_k)_{k=0}^{\infty} \subset C$ \\
& & & $C$: bounded set \\ \hline
\toprule
\end{tabular}
\end{table}

\section{Mathematical Preliminaries}\label{sec:preliminaries}
Let $\real$ be the set of all real numbers and $\nat$ be the set of all natural numbers (i.e., positive integers).
We denote $[n] := \{1,2,\cdots,n\}$ $(n \in \nat)$, $\nat_0 := \nat \cup \{0\}$ and $\real_{++} := \{x \in \real : x > 0\}$.
Let $M$ be a Riemannian manifold and $T_xM$ be a tangent space at $x \in M$. An exponential map at $x \in M$,
written as $\Exp_x:T_xM \to M$, is a mapping from $T_xM$ to $M$ with the requirement that a vector $\xi_x \in T_xM$ is mapped to the point $y := \Exp_x(\xi_x) \in M$ such that
there exists a geodesic $c:[0,1] \to M$ that satisfies $c(0) = x$, $c(1) = y$, and $\dot{c}(0) = \xi_x$,
where $\dot{c}$ is the derivative of $c$ \cite{sakai1996riemannian}.
Let $\ip{\cdot}{\cdot}_x$ be a Riemannian metric at $x \in M$ and $\norm{\cdot}_x$ be the norm defined by the Riemann metric at $x \in M$.
Let $d(\cdot,\cdot) : M \times M \to \real_{++} \cup \{0\}$ be the distance function on $M$.
A complete simply connected Riemannian manifold of a nonpositive sectional curvature is called a Hadamard manifold.

\subsection{Riemannian stochastic optimization problem}
We define a Riemannian stochastic optimization problem and two standard conditions.
Given a data point $z$ in data domain $Z$, a Riemannian stochastic optimization problem provides a smooth loss function, $\ell(\cdot;z):M \rightarrow \real$.
We minimize the expected loss $f:M\rightarrow\real$ defined by
\begin{align}~\label{eq:objective}
f(x) := \epc_{z \sim \mathcal{D}}\left[\ell(x;z)\right] = \epc\left[\ell_\xi(x)\right],
\end{align}
where $\mathcal{D}$ is a probability distribution over $Z$, $\xi$ denotes a random variable with distribution function $P$, and $\epc[\cdot]$ denotes the expectation taken with respect to $\xi$.
We assume that an SFO exists such that, for a given $x \in M$, it returns stochastic gradient $\mathsf{G}_\xi(x)$ of function $f$ defined by \eqref{eq:objective},
where a random variable $\xi$ is supported on $\Xi$ independently of $x$.
In the discussions hereafter, standard conditions (C1) and (C2) are assumed:
\begin{enumerate}[label=(C\arabic*)]
\item Let $(x_k)_{k=0}^{\infty} \subset M$ be the sequence generated by the algorithm. For each iteration $k$,
\begin{align*}
\epc_{\xi_k}\left[\mathsf{G}_{\xi_k}(x_k)\right] = \grad f(x_k),
\end{align*}
where $\xi_0, \xi_1, \cdots$ are independent samples, and the random variable $\xi_k$ is independent of $(x_i)_{i=0}^k$. There exists a nonnegative constant $\sigma^2$ such that
\begin{align*}
\epc_{\xi_k}\left[ \norm{\mathsf{G}_{\xi_k}(x_k) - \grad f(x_k)}^2_{x_k} \right] \leq \sigma^2.
\end{align*}
\item For each iteration $k \in \nat_0$, the optimizer samples batch $B_k$ of size $b$ independently of $k$ and estimates the full gradient $\grad f$ as
\begin{align*}
\grad f_{B_k}(x_k) := \frac{1}{b}\sum_{i \in [b]}\mathsf{G}_{\xi_{k,i}}(x_k),
\end{align*}
where $\xi_{k,i}$ is a random variable generated by the $i$-th sampling in the $k$-th iteration.
\end{enumerate}

From (C1) and (C2), we immediately have
\begin{align}
\epc \left[ \grad f_{B_k}(x_k) \relmiddle{|} x_k \right] = \grad f(x_k), \label{eq:e_bgrad} \\
\epc \left[ \norm{\grad f_{B_k}(x_k) - \grad f(x_k)}^2_{x_k} \relmiddle{|} x_k \right] \leq \frac{\sigma^2}{b}. \label{eq:v_bgrad}
\end{align}

Bonnabel \cite{bonnabel2013stochastic} proposed using RSGD for solving Riemannian optimization problems.
In this paper, we use RSGD considering batch size, as shown in Algorithm \ref{alg:RSGD}.

\begin{algorithm}
\caption{Riemannian stochastic gradient descent \cite{bonnabel2013stochastic}. 
\label{alg:RSGD}}
\begin{algorithmic}[1]
\Require Initial point $x_0 \in M$, step sizes $(\alpha_k)_{k=0}^{\infty} \subset \real_{++}$, batch size $b \in \nat$.
\Ensure Sequence $(x_k)_{k=0}^{\infty} \subset M$.
\State $k \leftarrow 0$.
\Loop
\State $\eta_k := - \grad f_{B_k} (x_k) = -b^{-1}\sum_{i \in [b]}\mathsf{G}_{\xi_{k,i}}(x_k)$.
\State $x_{k+1} := \Exp_{x_k}(\alpha_k\eta_k)$.
\State $k \leftarrow k + 1$.
\EndLoop
\end{algorithmic}
\end{algorithm}

\subsection{Useful Lemma}
The following lemma plays an important role in our discussion of the convergence of Riemannian stochastic gradient descent on a Hadamard manifold in Section \ref{sec:analysis}.

\begin{lemma}\label{lem:sn_bgrad}
Suppose that (C1) and (C2) and $M$ define a Riemannian manifold and that $f : M \rightarrow \real$ is a smooth function on $M$. Then, the sequence $(x_k)_{k=0}^{\infty} \subset M$ generated by Algorithm \ref{alg:RSGD} satisfies
\begin{align*}
\epc\left[\norm{\grad f_{B_k}(x_k)}^2_{x_k}\relmiddle{|}x_k\right] \leq \frac{\sigma^2}{b} + \norm{\grad f(x_k)}^2_{x_k}
\end{align*}
for all $k \in \nat_0$.
\end{lemma}
\begin{proof}
Using \eqref{eq:e_bgrad} and \eqref{eq:v_bgrad}, we obtain
\begin{align*}
\epc\left[\norm{\grad f_{B_k}(x_k)}^2_{x_k}\relmiddle{|}x_k\right] &= \epc\left[\norm{\grad f_{B_k}(x_k)-\grad f(x_k)+ f(x_k)}^2_{x_k}\relmiddle{|}x_k\right] \\
&= \epc\left[\norm{\grad f_{B_k}(x_k)-\grad f(x_k)}^2_{x_k}\relmiddle{|}x_k\right] \\
&\quad + 2\epc\left[ \ip{\grad f_{B_k}(x_k)-\grad f(x_k)}{\grad f(x_k)}_{x_k} \relmiddle{|}x_k\right] \\
&\quad + \epc\left[\norm{\grad f(x_k)}^2_{x_k} \relmiddle{|}x_k\right] \\
&\leq \frac{\sigma^2}{b} + \norm{\grad f(x_k)}^2_{x_k}
\end{align*}
for all $k \in \nat_0$.
\end{proof}

Lemma \ref{lem:seq-liminf} is useful in showing convergence of the limit inferior.
\begin{lemma}\label{lem:seq-liminf}
The sequences $(\alpha_k)_{k=0}^{\infty} \subset \real_{++}$ and  $(\beta_k)_{k=0}^{\infty} \subset \real$ such that
\begin{align*}
\sum_{k=0}^{+\infty}\alpha_k = +\infty, \quad \sum_{k=0}^{+\infty}\alpha_k\beta_k < +\infty
\end{align*}
satisfy
\begin{align*}
\liminf_{k \to +\infty} \beta_k \leq 0.
\end{align*}
\end{lemma}

Zhang and Sra developed the following lemma \cite{zhang2016first}.
\begin{lemma}\label{lem:Alexandrov}
Let $a$, $b$, and $c$ be the side lengths of a geodesic triangle in a Riemannian manifold with sectional curvature lower bounded by $\kappa$, and let $\theta$ be the angle between sides $b$ and $c$. Then,
\begin{align*}
a^2 \leq \zeta(\kappa, c)b^2 + c^2 - 2bc\cos(\theta),
\end{align*}
where $\zeta : \real \times \real_{++} \to \real_{++}$ is defined as
\begin{align*}
\zeta(\kappa, c) := \frac{\sqrt{\abs{\kappa}}c}{\tanh(\sqrt{\abs{\kappa}c})}.
\end{align*}
\end{lemma}

\section{Convergence of Riemannian Stochastic Gradient Descent on Hadamard Manifold}
\label{sec:analysis}
In this section, we describe four types of convergence analysis of Algorithm \ref{alg:RSGD} on a Hadamard manifold.

\subsection{Convergence of Riemannian Stochastic Gradient Descent with $L$-smoothness}
\label{sec:analysis-cns}
First, we describe the convergence of Algorithm \ref{alg:RSGD} when $L$-smoothness is assumed.
We start by defining the $L$-smoothness of the smooth function \cite{zhang2016first, huang2015broyden}.

\begin{definition}[smoothness]~\label{dfn:Lip} Let $M$ be a Hadamard manifold and $f : M \rightarrow \real$ be a smooth function on $M$.
For a positive number $L \in \real_{++}$, $f$ is said to be geodesically $L$-smooth if for any $x, y \in M$,
\begin{align*}
\norm{\grad f(x) - \Gamma_y^x(\grad f(y))}_x \leq L \norm{\Exp^{-1}_x(y)}_x,
\end{align*}
where $\Gamma_y^x : T_yM \to T_xM$ is the parallel transport from $y$ to $x$.
\end{definition}

We state the following lemma giving the necessary conditions for $L$-smooth \cite{zhang2016first}.
Lemma \ref{lem:Lip} plays an important role in convergence analysis with $L$-smoothness.

\begin{lemma}~\label{lem:Lip}
Let $M$ be a Hadamard manifold and $f : M \rightarrow \real$ be a smooth function on $M$. If $f$ is geodesically $L$-smooth, it follows that for all $x, y \in M$,
\begin{align*}
f(y) &\leq f(x) + \ip{\grad f(x)}{\Exp_x^{-1}(y)}_x + \frac{L}{2}\norm{\Exp_x^{-1}(y)}_x.
\end{align*}
\end{lemma}

To show the main result of this section (i.e., Theorems \ref{thm:Lip-cnv-cns} and \ref{thm:Lip-cnv-dim}), we present the following lemma, which plays a central role.

\begin{lemma}~\label{lem:Lip-cnv}
Let $M$ be a Hadamard manifold and $f : M \rightarrow \real$ be a smooth function.
We assume that $f$ is geodesically $L$-smooth and bounded below by $f_\star \in \real$.
Then, the sequence $(x_k)_{k=0}^{\infty} \subset M$ generated by Algorithm \ref{alg:RSGD} satisfies
\begin{align*}
\sum_{k = 0}^{K-1}\alpha_k\left(1-\frac{L\alpha_k}{2}\right)\epc\left[\norm{\grad f(x_k)}^2_{x_k}\right] \leq f(x_0)-f_\star + \frac{L\sigma^2}{2b}\sum_{k=0}^{K-1}\alpha_k^2
\end{align*}
for all $K \in \nat$.
\end{lemma}
\begin{proof}
From the $L$-smoothness of $f$, and $x_{k+1} = \Exp_{x_k}(\alpha_k\eta_k)$, we have
\begin{align}
f(x_{k+1}) \leq f(x_k) -\alpha_k \ip{\grad f(x_k)}{\grad f_{B_k}(x_k)}_{x_k} + \frac{L\alpha_k^2}{2}\norm{\grad f_{B_k}(x_k)}^2_{x_k} \label{eq:tmp_Ls}
\end{align}
for all $k \in \nat_0$. From \eqref{eq:e_bgrad}, we obtain
\begin{align}~\label{eq:tmp-ip-g-gb}
\epc\left[\ip{\grad f(x_k)}{\grad f_{B_k}(x_k)}_{x_k}\relmiddle{|}x_k\right] &= \ip{\grad f(x_k)}{\epc\left[\grad f_{B_k}(x_k)\relmiddle{|}x_k\right]}_{x_k} \nonumber \\
&= \norm{\grad f(x_k)}^2_{x_k}
\end{align}
for all $k \in \nat_0$. Hence, by taking $\epc[\cdot|x_k]$ of both sides of \eqref{eq:tmp_Ls}, we obtain
\begin{align*}
\epc\left[f(x_{k+1})\relmiddle{|} x_k \right] &\leq \epc\left[f(x_k)\relmiddle{|} x_k\right] - \alpha_k \epc\left[\ip{\grad f(x_k)}{\grad f_{B_k}(x_k)}_{x_k}\relmiddle{|} x_k\right] \\
&\quad + \frac{L\alpha_k^2}{2}\epc\left[\norm{\grad f_{B_k}(x_k)}^2_{x_k}\relmiddle{|} x_k\right] \\
&\leq f(x_k) -\alpha_k\norm{\grad f(x_k)}^2 + \frac{L\alpha_k^2}{2}\left(\frac{\sigma^2}{b} + \norm{\grad f(x_k)}^2_{x_k}\right) \\
&= f(x_k) + \left(\frac{L\alpha_k}{2}-1\right)\alpha_k\norm{\grad f(x_k)}_{x_k}^2 + \frac{L\sigma^2\alpha_k^2}{2b}
\end{align*}
for all $k \in \nat_0$, where the second inequality comes from \eqref{eq:tmp-ip-g-gb} and Lemma \ref{lem:sn_bgrad}.
Moreover, by taking $\epc[\cdot]$ of both sides, we obtain
\begin{align*}
\epc\left[f(x_{k+1})\right] \leq \epc\left[f(x_k)\right] + \left(\frac{L\alpha_k}{2}-1\right)\alpha_k\epc\left[\norm{\grad f(x_k)}_{x_k}^2\right] + \frac{L\sigma^2\alpha_k^2}{2b}
\end{align*}
for all $k \in \nat_0$. By summing up the above inequalities from $k = 0$ to $k=K-1$, we obtain
\begin{align*}
\sum_{k=0}^{K-1}\left(1-\frac{L\alpha_k}{2}\right)\alpha_k\epc\left[\norm{\grad f(x_k)}_{x_k}^2\right] &\leq f(x_0)-\epc[f(x_K)] + \frac{L\sigma^2}{2b}\sum_{k=0}^{K-1}\alpha_k^2 \\
&\leq f(x_0)-f_\star + \frac{L\sigma^2}{2b}\sum_{k=0}^{K-1}\alpha_k^2
\end{align*}
for all $K \in \nat$. This completes the proof.
\end{proof}

Using Lemma \ref{lem:Lip-cnv}, we present a convergence analysis with a constant step size and the assumption of $L$-smoothness.

\begin{theorem}~\label{thm:Lip-cnv-cns}
Let $M$ be a Hadamard manifold and $f : M \rightarrow \real$ be a smooth function.
We assume that $f$ is geodesically $L$-smooth and bounded below by $f_\star \in \real$.
If a constant step size $\alpha_k := \alpha$ $(k \in \nat_0)$ satisfies $0 < \alpha < 2 / L$, the sequence $(x_k)_{k=0}^{\infty} \subset M$ generated by Algorithm \ref{alg:RSGD} satisfies that for some $C_1,C_2 \in \real_{++}$,
\begin{align}~\label{eq:Lip-cnv-cns-sum}
\frac{1}{K}\sum_{k=0}^{K-1}\epc\left[\norm{\grad f(x_k)}_{x_k}^2\right] \leq \frac{C_1}{K} + \frac{C_2\sigma^2}{b},
\end{align}
for all $K \in \nat$.
\end{theorem}
\begin{proof}
From Lemma \ref{lem:Lip-cnv}, we have
\begin{align*}
\sum_{k = 0}^{K-1}\alpha\left(1-\frac{L\alpha}{2}\right)\epc\left[\norm{\grad f(x_k)}_{x_k}^2\right] \leq f(x_0)-f_\star + \frac{L\sigma^2}{2b}\sum_{k=0}^{K-1}\alpha^2
\end{align*}
for all $K \in \nat$. Moreover, from $0 < \alpha < 2 / L$, we have
\begin{align*}
0 < \frac{L\alpha}{2} < 1,
\end{align*}
which implies that
\begin{align*}
\frac{1}{K}\sum_{k=0}^{K-1}\epc\left[\norm{\grad f(x_k)}_{x_k}^2\right] \leq \underbrace{\frac{2(f(x_0)-f_\star)}{(2-L\alpha)\alpha}}_{C_1} \cdot \frac{1}{K} + \underbrace{\frac{L\alpha}{(2-L\alpha)}}_{C_2} \cdot \frac{\sigma^2}{b},
\end{align*}
for all $K \in \nat$. This completes the proof.
\end{proof}

Using Lemma \ref{lem:Lip-cnv}, we present a convergence analysis with a diminishing step size and the assumption of $L$-smoothness.

\begin{theorem}~\label{thm:Lip-cnv-dim}
Let $M$ be a Hadamard manifold and $f : M \rightarrow \real$ be a smooth function.
We assume that $f$ is geodesically $L$-smooth and bounded below by $f_\star \in \real$.
We consider a diminishing step size $(\alpha_k)_{k=0}^{\infty} \subset \real_{++}$ that satisfies
\begin{align}~\label{eq:diminishing-rule}
\sum_{k=0}^{+\infty} \alpha_k = +\infty, \quad \sum_{k=0}^{+\infty} \alpha_k^2 < +\infty.
\end{align}
Then, the sequence $(x_k)_{k=0}^{\infty} \subset M$ generated by Algorithm \ref{alg:RSGD} satisfies
\begin{align}~\label{eq:Lip-cnv-dim-linimf}
\liminf_{k \to +\infty} \epc\left[\norm{\grad f(x_k)}_{x_k}\right] = 0.
\end{align}
If a diminishing step size $(\alpha_k)_{k=0}^{\infty} \subset (0,1)$ is monotonically decreasing\footnote{In this case, $(\alpha_k)_{k=0}^{\infty} \subset (0,1)$ need not satisfy the conditions \eqref{eq:diminishing-rule}.},
then for all $K \in \nat$,
\begin{align} \label{eq:Lip-cnv-dim-sum}
\frac{1}{K}\sum_{k=0}^{K-1}\epc\left[\norm{\grad f(x_k)}_{x_k}^2\right] \leq \left(C_1 + \frac{C_2\sigma^2}{b}\sum_{k=0}^{K-1}\alpha_k^2\right)\frac{1}{K\alpha_{K-1}}
\end{align}
for some $C_1,C_2 \in \real_{++}$.
\end{theorem}
\begin{proof}
From \eqref{eq:diminishing-rule}, we obtain
\begin{align}~\label{eq:tmp-seq-a}
\sum_{k = 0}^{+\infty}\alpha_k\left(1-\frac{L\alpha_k}{2}\right) = +\infty.
\end{align}
In addition, from \eqref{eq:diminishing-rule}, $(\alpha_k)_{k=0}^{\infty}$ satisfies $\alpha_k \to 0$ $(k \to +\infty)$. This implies that there exists a natural number $k_0 \in \nat_0$ such that, for all $k \in \nat_0$, if $k \geq k_0$, then $0 < \alpha_k < 2 / L$. Therefore, we obtain
\begin{align*}
0 < 1-\frac{L\alpha_k}{2} < 1,
\end{align*}
which, together with Lemma \ref{lem:Lip-cnv}, means that
\begin{align}~\label{eq:tmp-bound}
&\sum_{k = k_0}^{K-1}\alpha_k\left(1-\frac{L\alpha_k}{2}\right)\epc\left[\norm{\grad f(x_k)}_{x_k}^2\right] \nonumber \\
&\leq f(x_0)-f_\star + \frac{L\sigma^2}{2b}\sum_{k=0}^{K-1}\alpha_k^2 - \sum_{k = 0}^{k_0 - 1}\alpha_k\left(1-\frac{L\alpha_k}{2}\right)\epc\left[\norm{\grad f(x_k)}_{x_k}^2\right]
\end{align}
for all $K \geq k_0 + 1$.
This implies
\begin{align}~\label{eq:tmp-seq-b}
\sum_{k = 0}^{+\infty}\alpha_k\left(1-\frac{L\alpha_k}{2}\right)\epc\left[\norm{\grad f(x_k)}_{x_k}^2\right] < +\infty.
\end{align}
By applying Lemma \ref{lem:seq-liminf} with \eqref{eq:tmp-seq-a} and \eqref{eq:tmp-seq-b}, we have that
\begin{align*}
\liminf_{k \to +\infty} \epc\left[\norm{\grad f(x_k)}_{x_k}^2\right] \leq 0,
\end{align*}
which, together with the convexity of $\norm{\cdot}_{x_k}$, means that
\begin{align*}
0 \leq \liminf_{k \to +\infty} \left(\epc\left[\norm{\grad f(x_k)}_{x_k}\right]\right)^2 \leq \liminf_{k \to +\infty} \epc\left[\norm{\grad f(x_k)}_{x_k}^2\right] \leq 0.
\end{align*}
This implies that
\begin{align*}
\liminf_{k \to +\infty} \epc\left[\norm{\grad f(x_k)}_{x_k}\right] = 0
\end{align*}
and ensures that \eqref{eq:Lip-cnv-dim-linimf} follows from the above discussion.

Furthermore, we show that \eqref{eq:Lip-cnv-dim-sum} follows from the above discussion. From the monotonicity of $(\alpha_k)_{k=0}^{\infty} \subset (0,1)$ and \eqref{eq:tmp-bound}, we have
\begin{align*}
&\alpha_{K-1}\left(1-\frac{L\alpha_{k_0}}{2}\right)\sum_{k = k_0}^{K-1}\epc\left[\norm{\grad f(x_k)}_{x_k}^2\right] \\
&\leq f(x_0)-f_\star + \frac{L\sigma^2}{2b}\sum_{k=0}^{K-1}\alpha_k^2 + \sum_{k = 0}^{k_0 - 1}\alpha_k\left(\frac{L\alpha_k}{2}-1\right)\epc\left[\norm{\grad f(x_k)}_{x_k}^2\right] \\
&\leq f(x_0)-f_\star + \frac{L\sigma^2}{2b}\sum_{k=0}^{K-1}\alpha_k^2 + \sum_{k = 0}^{k_0 - 1}L\alpha_k^2\epc\left[\norm{\grad f(x_k)}_{x_k}^2\right]
\end{align*}
for all $K\in\nat$. Hence, for all $K\in\nat$,
\begin{align*}
\sum_{k = k_0}^{K-1}\epc\left[\norm{\grad f(x_k)}_{x_k}^2\right] &\leq \frac{2\left(f(x_0)-f_\star+\sum_{k=0}^{k_0-1}L\alpha_k^2\epc\left[\norm{\grad f(x_k)}_{x_k}^2\right]\right)}{(2-L\alpha_{k_0})\alpha_{K-1}} \\
&\quad + \frac{L\sigma^2}{b(2-L\alpha_{k_0})\alpha_{K-1}}\sum_{k=0}^{K-1}\alpha_k^2,
\end{align*}
which, together with $\alpha_k \in (0,1)$ $(k \in \nat_0)$, gives us
\begin{align*}
&\frac{1}{K}\sum_{k = k_0}^{K-1}\epc\left[\norm{\grad f(x_k)}_{x_k}^2\right] \\
&\leq \underbrace{\left\{\frac{2\left(f(x_0)-f_\star+\sum_{k=0}^{k_0-1}L\alpha_k^2\epc\left[\norm{\grad f(x_k)}_{x_k}^2\right]\right)}{2-L\alpha_{k_0}} 
+ \sum_{k=0}^{k_0-1}\epc\left[\norm{\grad f(x_k)}_{x_k}^2\right]\right\}}_{C_1}\frac{1}{K\alpha_{K-1}} \\
& \quad + \underbrace{\frac{L}{2-L\alpha_{k_0}}}_{C_2} \cdot \frac{\sigma^2}{bK\alpha_{K-1}}\sum_{k=0}^{K-1}\alpha_k^2
\end{align*}
for all $K \in \nat$. This completes the proof.
\end{proof}

\subsection{Convergence of Riemannian Stochastic Gradient Descent without $L$-smoothness}
\label{sec:analysis-dim}
Next, we describe the convergence of Algorithm \ref{alg:RSGD} without the assumption of $L$-smoothness.
We start by reconsidering the definition of convergence of Algorithm \ref{alg:RSGD}.
N{\'e}meth \cite{nemeth2003variational} developed the variational inequality problem on a Hadamard manifold.
Motivated by the variational inequality problem in Euclidean space \cite{iiduka2022critical, iiduka2022theoretical}, we consider the following proposition.

\begin{proposition}\label{prp:variational}
Let $M$ be a Riemannian manifold and $f : M \rightarrow \real$ be a smooth function. Then, a stationary point $x \in M$ of $f$ satisfies
\begin{align*}
\grad f(x) = 0 \Leftrightarrow \forall y \in V_x, \, \ip{\grad f(x)}{-\Exp^{-1}_x(y)}_x \leq 0,
\end{align*}
where $V_x$ is a neighborhood of $x \in M$ such that $\Exp_x^{-1} : V_x \rightarrow T_xM$ is defined.
\end{proposition}
\begin{proof}
If $x \in M$ satisfies $\grad f(x) = 0$, we have
\begin{align*}
\ip{\grad f(x)}{-\Exp_x^{-1}(y)}_x &= \ip{0_x}{-\Exp_x^{-1}(y)}_x \leq 0
\end{align*}
for all $y \in V_x$. We assume that $x \in M$ satisfies $\ip{\grad f(x)}{-\Exp_x^{-1}(y)}_x \leq 0$ for all $y \in M$.
Let $y := \Exp_x(-\varepsilon\grad f(x))$, for which we choose a sufficiently small $\varepsilon > 0$ such that $y \in V_x$. We then have
\begin{align*}
\ip{\grad f(x)}{-\Exp_x^{-1}(y)}_x &= \ip{\grad f(x)}{-\Exp_x^{-1}(\Exp_x(-\varepsilon\grad f(x)))}_x \\
&= \ip{\grad f(x)}{\varepsilon\grad f(x)}_x \\
&= \varepsilon\norm{\grad f(x)}^2_x \leq 0.
\end{align*}
This implies that $\grad f(x) = 0$ and completes the proof.
\end{proof}

Note that, for a Hadamard manifold, $V_x = M$. From Proposition \ref{prp:variational}, we use the performance measure of the sequence $(x_k)_{k = 0}^\infty$,
\begin{align*}
V_k(x) := \epc\left[ \ip{\grad f(x_k)}{-\Exp^{-1}_{x_k}(x)}_{x_k} \right],
\end{align*}
for all $x \in M$. In practice, we use $V_k(x)$ for showing the convergence of Algorithm \ref{alg:RSGD} in Theorems \ref{thm:nLip-cnv-cns} and \ref{thm:nLip-cnv-dim}.
To show the main result of this section (i.e., Theorems \ref{thm:nLip-cnv-cns} and \ref{thm:nLip-cnv-dim}), we present the following lemma, which plays a central role.

\begin{lemma}~\label{lem:nLip-cnv}
Let $M$ be a Hadamard manifold with sectional curvature lower bounded by $\kappa$ and $f : M \rightarrow \real$ be a smooth function.
Then, the sequence $(x_k)_{k=0}^{\infty} \subset M$ generated by Algorithm \ref{alg:RSGD} satisfies that, for all $K \in \nat$ and $x \in M$,
\begin{align*}
&\sum_{k=0}^{K-1}\alpha_k V_k(x) \leq \frac{\zeta\left(\kappa, D(x)\right)}{2}\sum_{k=0}^{K-1}\alpha_k^2\left(\frac{\sigma^2}{b} + \epc\left[\norm{\grad f(x_k)}^2_{x_k}\right]\right) \\
&\quad + \frac{1}{2}\left(\epc\left[\norm{\Exp^{-1}_{x_0}(x)}_{x_0}^2\right] - \epc\left[\norm{\Exp^{-1}_{x_K}(x)}_{x_K}^2\right] \right),
\end{align*}
where $\zeta : \real_{++} \times \real_{++} \to \real_{++}$ is defined as Lemma \ref{lem:Alexandrov}.
\end{lemma}
\begin{proof}
For arbitrary $x \in M$, we consider a geodesic triangle consisting of three points, $x_k$, $x_{k+1}$, and $x$. Let the length of each side be $a$, $b$, and $c$, respectively, such that
\begin{align}~\label{eq:abc}
\begin{cases}
a := d(x_{k+1}, x) \\
b := d(x_k, x_{k+1}) = \alpha_k\norm{\grad f_{B_k}(x_k)}_{x_k} \\
c := d(x_k, x).
\end{cases}
\end{align}
Let $\theta \in \real$ be the angle between sides $b$ and $c$. It then follows that
\begin{align*}
\cos(\theta) := \frac{\ip{\Exp^{-1}_{x_k}(x_{k+1})}{\Exp^{-1}_{x_k}(x)}_{x_k}}{\norm{\Exp^{-1}_{x_k}(x_{k+1})}_{x_k} \norm{\Exp^{-1}_{x_k}(x)}_{x_k}}.
\end{align*}
From Lemma \ref{lem:Alexandrov} with \eqref{eq:abc}, we have
\begin{align*}
\norm{\Exp^{-1}_{x_{k+1}}(x)}_{x_{k+1}}^2 &\leq \alpha_k^2\zeta\left(\kappa, D(x)\right)\norm{\grad f_{B_k}(x_k)}^2_{x_k} \\
& - 2\alpha_k\ip{\grad f_{B_k}(x_k)}{-\Exp^{-1}_{x_k}(x)}_{x_k} + \norm{\Exp^{-1}_{x_k}(x)}^2_{x_k}.
\end{align*}
By taking $\epc[\cdot|x_k]$ of both sides of this inequation, we obtain
\begin{align*}
\epc\left[\norm{\Exp^{-1}_{x_{k+1}}(x)}_{x_{k+1}}^2\relmiddle{|}x_k\right]
&\leq \alpha_k^2\zeta\left(\kappa, D(x)\right)\epc\left[\norm{\grad f_{B_k}(x_k)}^2_{x_k}\relmiddle{|}x_k\right] \\
& - 2\alpha_k\epc\left[\ip{\grad f_{B_k}(x_k)}{-\Exp^{-1}_{x_k}(x)}_{x_k}\relmiddle{|}x_k\right] \\
& + \epc\left[\norm{\Exp^{-1}_{x_k}(x)}^2_{x_k}\relmiddle{|}x_k\right] \\
&\leq \alpha_k^2\zeta\left(\kappa, D(x)\right)\left(\frac{\sigma^2}{b} + \norm{\grad f(x_k)}^2_{x_k}\right) \\
& - 2\alpha_k\ip{\grad f(x_k)}{-\Exp^{-1}_{x_k}(x)}_{x_k} \\
& + \norm{\Exp^{-1}_{x_k}(x)}^2_{x_k}
\end{align*}
for all $k \in \nat_0$, where the second inequality comes from Lemma \ref{lem:sn_bgrad}. Furthermore, by taking $\epc[\cdot]$ of both sides, we obtain
\begin{align}~\label{eq:tmp-e-triangle}
\epc\left[\norm{\Exp^{-1}_{x_{k+1}}(x)}_{x_{k+1}}^2\right] &\leq \alpha_k^2\zeta\left(\kappa, D(x)\right)\left(\frac{\sigma^2}{b} + \epc\left[\norm{\grad f(x_k)}^2_{x_k}\right]\right) \nonumber \\
& - 2\alpha_kV_k(x) + \epc\left[\norm{\Exp^{-1}_{x_k}(x)}^2_{x_k}\right]
\end{align}
for all $k \in \nat_0$. Hence,
\begin{align*}
\alpha_k V_k(x) &\leq \frac{1}{2}\left(\epc\left[\norm{\Exp^{-1}_{x_k}(x)}_{x_k}^2\right] - \epc\left[\norm{\Exp^{-1}_{x_{k+1}}(x)}_{x_{k+1}}^2\right] \right) \\
& + \frac{\alpha_k^2\zeta\left(\kappa, D(x)\right)}{2}\left(\frac{\sigma^2}{b} + \epc\left[\norm{\grad f(x_k)}^2_{x_k}\right]\right)
\end{align*}
for all $k \in \nat_0$. By summing up the above inequalities from $k = 0$ to $k=K-1$ $(K \in \nat)$, we obtain
\begin{align*}
\sum_{k=0}^{K-1}\alpha_k V_k(x) &\leq \frac{\zeta\left(\kappa, D(x)\right)}{2}\sum_{k=0}^{K-1}\alpha_k^2\left(\frac{\sigma^2}{b} + \epc\left[\norm{\grad f(x_k)}^2_{x_k}\right]\right) \\
& + \frac{1}{2}\left(\epc\left[\norm{\Exp^{-1}_{x_0}(x)}_{x_0}^2\right] - \epc\left[\norm{\Exp^{-1}_{x_K}(x)}_{x_K}^2\right] \right)
\end{align*}
for all $K \in \nat$, and this completes the proof.
\end{proof}

Here, we make the following assumptions:

\begin{assum}~\label{asm:egrad-bound}
Let $M$ be a Hadamard manifold, $f : M \rightarrow \real$ be a smooth function, and $(x_k)_{k=0}^{\infty} \subset M$ be a sequence generated by Algorithm \ref{alg:RSGD}.
\begin{enumerate}[label=(A\arabic*)]
\item We assume that there exists a positive number $G \in \real_{++}$ such that
\begin{align*}
\epc\left[\norm{\grad f(x_k)}_{x_k}\right] \leq G < +\infty
\end{align*}
for all $k \in \nat_0$.
\item We define $D: M \to \real$ as
\begin{align*}
D(x) := \sup\left\{\epc\left[d(x_k,x)\right] \in \real_{++} : k \in \nat_0 \right\}
\end{align*}
and assume that $D(x) < + \infty$ for all $x \in M$.
\end{enumerate}
\end{assum}

Using Lemma \ref{lem:nLip-cnv}, we present a convergence analysis with a constant step size.

\begin{theorem}~\label{thm:nLip-cnv-cns}
Suppose Assumption \ref{asm:egrad-bound}, and let $M$ be a Hadamard manifold with sectional curvature lower bounded by $\kappa$ and $f : M \rightarrow \real$ be a smooth function.
If we use a constant step size $\alpha_k := \alpha > 0$ $(k \in \nat_0)$, the sequence $(x_k)_{k=0}^{\infty} \subset M$ generated by Algorithm \ref{alg:RSGD} satisfies
\begin{align}\label{eq:nLip-cnv-cns-liminf}
\liminf_{k \to +\infty} V_k(x) \leq \left(\frac{\sigma^2}{b} + G^2\right)\alpha C
\end{align}
for some $C \in \real_{++}$. Moreover, for all $K \in \nat$ and $x \in M$,
\begin{align}~\label{eq:nLip-cnv-cns-sum}
\frac{1}{K}\sum_{k=0}^{K-1} V_k(x) \leq \left(\frac{\sigma^2}{b} + G^2\right)\alpha C_1 + \frac{C_2}{K}
\end{align}
for some $C_1,C_2\in\real_{++}$.
\end{theorem}
\begin{proof}
If $\grad f(x_{k_0}) = 0$ for some $k_0 \in \nat_0$, then \eqref{eq:nLip-cnv-cns-liminf} follows.
Thus, it is sufficient to prove \eqref{eq:nLip-cnv-cns-liminf} only when $\grad f(x_k) \neq 0$ for all $k \in \nat_0$.
We assume that there exists a positive number $\varepsilon \in \real_{++}$ such that
\begin{align}~\label{eq:tmp-contradiction}
\liminf_{k \to +\infty} V_k(x) > \frac{\zeta\left(\kappa, D(x)\right)}{2}\left(\frac{\sigma^2}{b} + G^2\right)\alpha + \varepsilon
\end{align}
for all $x \in M$. Furthermore, from the definition of the limit inferior, there exists $k_0 \in \nat_0$ such that, for all $k \geq k_0$,
\begin{align*}
\liminf_{k \to +\infty} V_k(x) -\frac{\varepsilon}{2} < V_k(x),
\end{align*}
from which, together with \eqref{eq:tmp-contradiction}, we obtain
\begin{align*}
V_k(x) > \frac{\zeta\left(\kappa, D(x)\right)}{2}\left(\frac{\sigma^2}{b} + G^2\right)\alpha + \frac{\varepsilon}{2}
\end{align*}
for all $k \geq k_0$.
Here, from \eqref{eq:tmp-e-triangle} and $\alpha_k = \alpha$ $(k \in \nat_0)$, then for all $k \geq k_0$,
\begin{align*}
\epc\left[\norm{\Exp^{-1}_{x_{k+1}}(x)}_{x_{k+1}}^2\right]
&\leq \alpha^2\zeta\left(\kappa, D(x)\right)\left(\frac{\sigma^2}{b} + \epc\left[\norm{\grad f(x_k)}^2_{x_k}\right]\right) \\
& - 2\alpha V_k(x) + \epc\left[\norm{\Exp^{-1}_{x_k}(x)}^2_{x_k}\right], \\
& < \epc\left[\norm{\Exp^{-1}_{x_k}(x)}^2_{x_k}\right], \\
& + \alpha^2\zeta\left(\kappa, D(x)\right)\left(\frac{\sigma^2}{b} + G^2\right) \\
& - 2\alpha \left\{ \frac{\zeta\left(\kappa, D(x)\right)}{2}\left(\frac{\sigma^2}{b} + G^2\right)\alpha + \frac{\varepsilon}{2} \right\} \\
& = \epc\left[\norm{\Exp^{-1}_{x_k}(x)}^2_{x_k}\right] -\alpha\varepsilon.
\end{align*}
Hence,
\begin{align}\label{eq:tmp-edist-bound}
0 \leq \epc\left[\norm{\Exp^{-1}_{x_{k+1}}(x)}_{x_{k+1}}^2\right] < \epc\left[\norm{\Exp^{-1}_{x_{k_0}}(x)}^2_{x_{k_0}}\right] -\alpha\varepsilon (k + 1 - k_0).
\end{align}
When $k$ diverges to $+\infty$, the right side of \eqref{eq:tmp-edist-bound} diverges to $-\infty$. By contradiction, we have
\begin{align*}
\liminf_{k \to +\infty} V_k(x) \leq \underbrace{\frac{\zeta\left(\kappa, D(x)\right)}{2}}_{C}\left(\frac{\sigma^2}{b} + G^2\right)\alpha.
\end{align*}
Next, we show the upper bound of $(1/K)\sum_{k=0}^{K-1}V_k(x)$ such as expressed by \eqref{eq:nLip-cnv-cns-sum}.
From Lemma \ref{lem:nLip-cnv} with $\alpha_k = \alpha$ $(k \in \nat_0)$, we have that, for all $K \in \nat$,
\begin{align*}
\sum_{k=0}^{K-1}\alpha V_k(x) & \leq \frac{\zeta\left(\kappa, D(x)\right)}{2}\sum_{k=0}^{K-1}\alpha^2\left(\frac{\sigma^2}{b} + \epc\left[\norm{\grad f(x_k)}^2_{x_k}\right]\right) \\
&\quad + \frac{1}{2}\left(\epc\left[\norm{\Exp^{-1}_{x_0}(x)}_{x_0}^2\right] - \epc\left[\norm{\Exp^{-1}_{x_K}(x)}_{x_K}^2\right] \right) \\
& \leq \frac{\zeta\left(\kappa, D(x)\right)}{2}\left(\frac{\sigma^2}{b} + G^2\right)K\alpha^2 + \frac{D(x)}{2},
\end{align*}
which implies
\begin{align*}
\frac{1}{K}\sum_{k=0}^{K-1} V_k(x) \leq \underbrace{\frac{\zeta\left(\kappa, D(x)\right)}{2}}_{C_1}\left(\frac{\sigma^2}{b} + G^2\right)\alpha + \underbrace{\frac{D(x)}{2\alpha}}_{C_2}\cdot\frac{1}{K}
\end{align*}
for all $K \in \nat$. This completes the proof.
\end{proof}

Using Lemma \ref{lem:nLip-cnv}, we present a convergence analysis with a diminishing step size.

\begin{theorem}~\label{thm:nLip-cnv-dim}
Suppose Assumption \ref{asm:egrad-bound}, and let $M$ be a Hadamard manifold with sectional curvature lower bounded by $\kappa$ and $f : M \rightarrow \real$ be a smooth function.
If we use a diminishing step size $(\alpha_k)_{k=0}^{\infty} \subset \real_{++}$ such as \eqref{eq:diminishing-rule}, the sequence $(x_k)_{k=0}^{\infty} \subset M$ generated by Algorithm \ref{alg:RSGD} satisfies
\begin{align}~\label{eq:nLip-cnv-dim-liminf}
\liminf_{k \to +\infty} V_k(x) \leq 0.
\end{align}
If diminishing step size $(\alpha_k)_{k=0}^{\infty} \subset \real_{++}$ is monotone decreasing\footnote{In this case, $(\alpha_k)_{k=0}^{\infty} \subset \real_{++}$ need not satisfy the conditions \eqref{eq:diminishing-rule}.},
then for all $K \in \nat$ and $x \in M$,
\begin{align}~\label{eq:nLip-cnv-dim-sum}
\frac{1}{K}\sum_{k=0}^{K-1} V_k(x) \leq \left(\frac{\sigma^2}{b} + G^2\right)\frac{C_1}{K}\sum_{k=0}^{K-1}\alpha_k + \frac{C_2}{\alpha_{K-1}K}
\end{align}
for some $C_1,C_2 \in \real_{++}$.
\end{theorem}
\begin{proof}
Using Lemma \ref{lem:nLip-cnv}, for all $K \in \nat$ and $x \in M$, we obtain
\begin{align*}
\sum_{k=0}^{K-1}\alpha_k V_k(x) \leq \frac{\zeta\left(\kappa, D(x)\right)}{2}\left(\frac{\sigma^2}{b} + G^2\right)\sum_{k=0}^{K-1}\alpha_k^2 + \frac{1}{2}\epc\left[\norm{\Exp^{-1}_{x_0}(x)}_{x_0}^2\right],
\end{align*}
from which, together with $\sum_{k=0}^{+\infty}\alpha_k < +\infty$, we obtain
\begin{align}~\label{eq:tmp-seq-c}
\sum_{k=0}^{K-1}\alpha_k V_k(x) < + \infty
\end{align}
for all $K \in \nat$. From Lemma \ref{lem:seq-liminf} with $\sum_{k=0}^{+\infty}\alpha_k = + \infty$ and \eqref{eq:tmp-seq-c}, we have that
\begin{align*}
\liminf_{k \to +\infty} V_k(x) \leq 0.
\end{align*}
This means that \eqref{eq:nLip-cnv-dim-liminf} follows.

Next, we show the upper bound of $(1/K)\sum_{k=0}^{K-1}V_k(x)$ such as \eqref{eq:nLip-cnv-dim-sum}. From Lemma \ref{lem:nLip-cnv}, we obtain
\begin{align*}
\frac{1}{K}\sum_{k=0}^{K-1}V_k(x) & \leq \frac{\zeta\left(\kappa, D(x)\right)}{2K}\left(\frac{\sigma^2}{b} + G^2\right)\sum_{k=0}^{K-1}\alpha_k \\
& + \frac{1}{2K}\underbrace{\sum_{k=0}^{K-1}\frac{\epc\left[\norm{\Exp^{-1}_{x_k}(x)}_{x_k}^2\right] - \epc\left[\norm{\Exp^{-1}_{x_{k+1}}(x)}_{x_{k+1}}^2\right]}{\alpha_k}}_{X_K(x)}
\end{align*}
for all $K \in \nat$. Hence,
\begin{align*}
X_K(x) &:= \frac{\epc\left[\norm{\Exp^{-1}_{x_0}(x)}_{x_0}^2\right]}{\alpha_0} + \sum_{k=1}^{K-1}\left(\frac{\epc\left[\norm{\Exp^{-1}_{x_k}(x)}_{x_k}^2\right]}{\alpha_k} - \frac{\epc\left[\norm{\Exp^{-1}_{x_k}(x)}_{x_k}^2\right]}{\alpha_{k-1}} \right) \\
& \quad - \frac{\epc\left[\norm{\Exp^{-1}_{x_K}(x)}_{x_K}^2\right]}{\alpha_{K-1}} \\
& \leq \frac{D(x)}{\alpha_0} + D(x)\sum_{k=0}^{K-1}\left(\frac{1}{\alpha_k} - \frac{1}{\alpha_{k-1}}\right) \\
& = \frac{D(x)}{\alpha_{K-1}},
\end{align*}
where the second inequality comes from the monotonicity of $(\alpha_k)_{k=0}^{\infty}$.
Therefore, it follows that
\begin{align*}
\frac{1}{K}\sum_{k=0}^{K-1}V_k(x) \leq
\underbrace{\frac{\zeta\left(\kappa, D(x)\right)}{2}}_{C_1}\left(\frac{\sigma^2}{b} + G^2\right)\frac{1}{K}\sum_{k=0}^{K-1}\alpha_k + \underbrace{\frac{D(x)}{2}}_{C_2} \cdot \frac{1}{\alpha_{K-1}K}
\end{align*}
for all $K \in \nat$ and $x \in M$. This completes the proof.
\end{proof}

\subsection{Convergence Rate of Practical Step Sizes}
\label{sec:rate}
Now we specifically calculate the convergence rate for practical step sizes.

First, we consider the constant step size defined as $\alpha_k := \alpha \in \real_{++}$. From Theorems \ref{thm:Lip-cnv-cns} and \ref{thm:nLip-cnv-cns},
we immediately obtain the convergence rate for the constant step size. Hence, the convergence rates with and without $L$-smoothness are
\begin{align*}
\mathcal{O}\left(\dfrac{1}{K}+\dfrac{\sigma^2}{b}\right)\text{, and }\mathcal{O}\left(\dfrac{1}{K}\right) + \left(\dfrac{\sigma^2}{b} + G^2\right)\alpha,
\end{align*}
respectively.

Next, we consider the diminishing step size defined as $\alpha_k := 1 / \sqrt{k + 1}$.
Substituting $\alpha_k := 1 / \sqrt{k + 1}$ for the right side of \eqref{eq:Lip-cnv-dim-sum}, we obtain
\begin{align*}
\left(C_1 + \frac{C_2\sigma^2}{b}\sum_{k=0}^{K-1}\alpha_k^2\right)\frac{1}{K\alpha_{K-1}} \leq \left(C_1 + \frac{C_2\sigma^2}{b}(1 + \log K)\right)\frac{1}{\sqrt{K}},
\end{align*}
where
\begin{align*}
\sum_{k=0}^{K-1}\frac{1}{K+1} \leq 1 + \int_1^K\frac{dt}{t} = 1 + \log K.
\end{align*}
Substituting $\alpha_k := 1 / \sqrt{k + 1}$ for the right side of \eqref{eq:nLip-cnv-dim-sum}, we obtain 
\begin{align}
\left(\frac{\sigma^2}{b} + G^2\right)\frac{C_1}{K}\sum_{k=0}^{K-1}\alpha_k + \frac{C_2}{\alpha_{K-1}K} \leq C_1\left(\frac{\sigma^2}{b} + G^2\right)\left(\frac{2}{\sqrt{K}} - \frac{1}{K}\right) + \frac{C_2}{\sqrt{K}}, \label{eq:sfo-temp1}
\end{align}
where
\begin{align*}
\frac{1}{K}\sum_{k=0}^{K-1}\frac{1}{\sqrt{K+1}} \leq \frac{1}{K}\left( 1 + \int_1^K\frac{dt}{\sqrt{t}}\right) = \frac{2}{\sqrt{K}} - \frac{1}{K}.
\end{align*}
Therefore, the convergence rates of $\alpha_k := 1 / \sqrt{k + 1}$ with and without $L$-smoothness are
\begin{align*}
\mathcal{O}\left(\dfrac{\log K}{\sqrt{K}}\right)\text{, and }\mathcal{O}\left(\left(1+\dfrac{\sigma^2}{b}\right)\dfrac{1}{\sqrt{K}}\right), 
\end{align*}
respectively.

Finally, we consider the diminishing step size defined as $\alpha_k := \alpha \gamma^{p_k}$, where $\alpha,\gamma \in (0,1)$, $n, T \in \nat$ and
\begin{align*}
p_k := \min\left\{\max\left\{m \in \nat_0 : m \leq \frac{k}{T}\right\}, n\right\}.
\end{align*}
This step size is explicitly represented as
\begin{align*}
& \underbrace{\alpha, \alpha, \cdots, \alpha}_{T},
\underbrace{\alpha\gamma, \alpha\gamma, \cdots, \alpha\gamma}_{T},
\underbrace{\alpha\gamma^2, \alpha\gamma^2, \cdots, \alpha\gamma^2}_{T}, \\
& \quad \cdots, \underbrace{\alpha\gamma^{n-1}, \alpha\gamma^{n-1}, \cdots, \alpha\gamma^{n-1}}_{T},
\alpha\gamma^n, \alpha\gamma^n, \cdots, 
\end{align*}
which means that $0 < \alpha\gamma^n \leq \alpha_k \leq \alpha$ for all $k \in \nat_0$.
Substituting $\alpha_k := \alpha \gamma^{p_k}$ for the right side of \eqref{eq:Lip-cnv-dim-sum}, we obtain
\begin{align*}
\left(C_1 + \frac{C_2\sigma^2}{b}\sum_{k=0}^{K-1}\alpha_k^2\right)\frac{1}{K\alpha_{K-1}}
&\leq \frac{C_1}{\alpha\gamma^nK} + \frac{C_2\alpha\sigma^2}{\gamma^nb}.
\end{align*}
Substituting $\alpha_k := \alpha \gamma^{p_k}$ for the right side of \eqref{eq:nLip-cnv-dim-sum}, we obtain
\begin{align}
\left(\frac{\sigma^2}{b} + G^2\right)\frac{C_1}{K}\sum_{k=0}^{K-1}\alpha_k + \frac{C_2}{\alpha_{K-1}K} \leq \left(\frac{\sigma^2}{b} + G^2\right)C_1\alpha + \frac{C_2}{\alpha\gamma^nK}. \label{eq:sfo-temp2}
\end{align}
Therefore, the convergence rates of $\alpha_k := \alpha \gamma^{p_k}$ with and without $L$-smoothness are
\begin{align*}
\mathcal{O}\left(\dfrac{1}{K}+\dfrac{\sigma^2}{b}\right)\text{, and }\mathcal{O}\left(\dfrac{1}{K}\right) + \mathcal{O}\left(\left(\dfrac{\sigma^2}{b} + G^2\right)\alpha\right), 
\end{align*}
respectively.

We summarize the convergence rates of practical step sizes in Table \ref{tab:rate}
It shows that increasing the batch size improves RSGD performance and that the constant $\alpha$ in learning rates should be sufficiently small.
\begin{table}[htbp]
\caption{Convergence rates of three practical step sizes with and without assumption of $L$-smoothness ($\gamma \in (0, 1)$, $G, \alpha > 0$, $\sigma^2 \geq 0$, and $b$ is batch size). \label{tab:rate}}
\centering
\begin{tabular}{ccc}
\bottomrule
\hline
\multirow{2}{*}{Step size $\alpha_k$} &\multicolumn{2}{c}{Convergence rate} \\
\cmidrule(lr){2-3}
& with $L$-smooth & without $L$-smooth \\ \hline
\multirow{2}{*}{$\alpha_k = \alpha$} & \multirow{3}{*}{$\mathcal{O}\left(\dfrac{1}{K}+\dfrac{\sigma^2}{b}\right)$} & \multirow{3}{*}{$\mathcal{O}\left(\dfrac{1}{K}\right) + \left(\dfrac{\sigma^2}{b} + G^2\right)\alpha$} \\
 &  &  \\
(Constant) & & \\ \hline
 \multirow{2}{*}{$\alpha_k = 1/\sqrt{k+1}$} & \multirow{3}{*}{$\mathcal{O}\left(\dfrac{\log K}{\sqrt{K}}\right)$} & \multirow{3}{*}{$\mathcal{O}\left(\left(1+\dfrac{\sigma^2}{b}\right)\dfrac{1}{\sqrt{K}}\right)$} \\
 &  &  \\
(Diminishing) & & \\ \hline
\multirow{2}{*}{$\alpha_k = \alpha\gamma^{k}$} & \multirow{3}{*}{$\mathcal{O}\left(\dfrac{1}{K}+\dfrac{\sigma^2}{b}\right)$}  & \multirow{3}{*}{$\mathcal{O}\left(\dfrac{1}{K}\right) + \mathcal{O}\left(\left(\dfrac{\sigma^2}{b} + G^2\right)\alpha\right)$} \\
 &  &  \\
(Diminishing) & & \\ \hline
\toprule
\end{tabular}
\end{table}

\subsection{Existence of Critical Batch Size}
\label{sec:critical}
Motivated by the work of Zhang and others \cite{zhang2019algorithmic, shallue2019measuring, iiduka2022critical},
we use SFO complexity as the performance measure for a Riemannian stochastic optimizer.
We define SFO complexity as $Kb$, where $K$ is the number of steps needed for solving the problem, and $b$ is the batch size used in Algorithm \ref{alg:RSGD}.
Let $b^\star$ be the critical batch size for which $Kb$ is minimized.

Our analyses (i.e., Theorems \ref{thm:nLip-cnv-cns} and \ref{thm:nLip-cnv-dim}) lead to the number of steps $K$ needed to satisfy an $\varepsilon$-approximation, which is defined as
\begin{align*}
\frac{1}{K}\sum_{k=0}^{K-1}V_k(x) \leq \varepsilon.
\end{align*}

\begin{theorem}~\label{thm:critical}
Suppose Assumption \ref{asm:egrad-bound}, and let $M$ be a Hadamard manifold with sectional curvature lower bounded by $\kappa$ and $f : M \rightarrow \real$ be a smooth function.
Then, the numbers of steps $K$ needed to satisfy an $\varepsilon$-approximation for $\alpha_k = \alpha$, $\alpha_k = 1/\sqrt{k + 1}$, and $\alpha_k = \alpha\gamma^{p_k}$ are respectively
\begin{align*}
K &= \frac{C_2b}{\varepsilon b - (\sigma^2+G^2b)\alpha C_1}, \quad b > \frac{ (\sigma^2+G^2b)\alpha C_1}{\varepsilon}, \\
K &= \left(\frac{2C_1\sigma^2 + (2C_1G^2+C_2)b}{\varepsilon b}\right)^2 \\
K &= \frac{C_2b\alpha^{-1}\gamma^{-n}}{\varepsilon b - (\sigma^2+G^2b)\alpha C_1}, \quad b > \frac{ (\sigma^2+G^2b)\alpha C_1}{\varepsilon}.
\end{align*}
The $K$ is convex and monotone decreasing with respect to $b$. SFO complexity $Kb$ is convex with respect to $b$,
and there exist critical batch sizes $b^\star$ for $\alpha_k = \alpha$, $\alpha_k = 1/\sqrt{k + 1}$, and $\alpha_k = \alpha\gamma^{p_k}$ 
\begin{align*}
b^\star &= \frac{2C_2\sigma^2\alpha}{\varepsilon - G^2\alpha C_1}, \\
b^\star &= \exp\left\{\left(\frac{2C_1G^2 + C_2}{2C_1\sigma^2}\right)^2\right\}, \\
b^\star &= \frac{2C_2\sigma^2\alpha}{\varepsilon - G^2\alpha C_1}.
\end{align*}
\end{theorem}
\begin{proof}
First, we consider $\alpha_k = \alpha$.
From the upper bound of \eqref{eq:nLip-cnv-cns-sum}, let us consider
\begin{align*}
\left(\frac{\sigma^2}{b} + G^2\right)\alpha C_1 + \frac{C_2}{K} = \varepsilon,
\end{align*}
which implies
\begin{align*}
K = \frac{C_2b}{\varepsilon b - (\sigma^2+G^2b)\alpha C_1}, \quad b > \frac{ (\sigma^2+G^2b)\alpha C_1}{\varepsilon}.
\end{align*}
Since
\begin{align*}
\frac{dK}{db} &= -\frac{C_1C_2\alpha\sigma^2}{\{\varepsilon b - (\sigma^2+G^2b)\alpha C_1\}^2} \leq 0, \\
\frac{d^2K}{db^2} &= \frac{2C_1C_2\alpha\sigma^2(\varepsilon - G^2\alpha C_1)}{\{\varepsilon b - (\sigma^2+G^2b)\alpha C_1\}^3} \geq 0,
\end{align*}
where the second inequality comes from $\varepsilon = (\sigma^2 / b + G^2)\alpha C_1 + C_2/K > G^2\alpha C_1$,
and $K$ is convex and monotone decreasing with respect to batch size $b$.
Next, SFO complexity, defined as
\begin{align*}
Kb := \frac{C_2b^2}{\varepsilon b - (\sigma^2+G^2b)\alpha C_1},
\end{align*}
is convex with respect to batch size $b$ since
\begin{align*}
\frac{d^2(Kb)}{db^2} = \frac{2C_1^2C_2\alpha^2\sigma^4}{\{\varepsilon b - (\sigma^2+G^2b)\alpha C_1\}^3} \geq 0.
\end{align*}
Moreover, from
\begin{align*}
\frac{d(Kb)}{db} = \frac{C_2b(\varepsilon b - bG^2\alpha C_1 - 2 \sigma^2\alpha C_1)}{\{\varepsilon b - (\sigma^2+G^2b)\alpha C_1\}^2},
\end{align*}
$d(Kb)/db = 0$ if and only if $b = 2C_2\sigma^2\alpha / (\varepsilon - G^2\alpha C_1)$. Therefore, SFO complexity $Kb$ is minimized at
\begin{align*}
b^\star := \frac{2C_2\sigma^2\alpha}{\varepsilon - G^2\alpha C_1},
\end{align*}
which is the critical batch size.

Next, we consider $\alpha_k = 1 / \sqrt{k + 1}$.
From \eqref{eq:sfo-temp1}, let us consider
\begin{align*}
\frac{2C_1}{\sqrt{K}}\left(\frac{\sigma^2}{b} + G^2\right) + \frac{C_2}{\sqrt{K}} = \varepsilon,
\end{align*}
which implies
\begin{align*}
K = \left(\frac{2C_1\sigma^2 + (2C_1G^2+C_2)b}{\varepsilon b}\right)^2.
\end{align*}
Moreover, from
\begin{align*}
\frac{d(Kb)}{db} = -\left(\frac{2C_1\sigma^2}{\varepsilon}\right)^2\log b + \left(\frac{2C_1G^2+C_2}{\varepsilon}\right)^2,
\end{align*}
$d(Kb)/db = 0$ if and only if $\log b = \{(2C_1G^2 + C_2) / (2C_1\sigma^2)\}^2$.
Therefore, SFO complexity $Kb$ is minimized at
\begin{align*}
b^\star := \exp\left\{\left(\frac{2C_1G^2 + C_2}{2C_1\sigma^2}\right)^2\right\},
\end{align*}
which again is the critical batch size.

Finally, we consider $\alpha_k = \alpha\gamma^{p_k}$.
From \eqref{eq:sfo-temp2}, let us consider
\begin{align*}
\left(\frac{\sigma^2}{b} + G^2\right)C_1\alpha + \frac{C_2}{\alpha\gamma^nK} = \varepsilon,
\end{align*}
which implies
\begin{align*}
K = \frac{1}{\alpha\gamma^n} \cdot \frac{C_2b}{\varepsilon b - (\sigma^2+G^2b)\alpha C_1}, \quad b > \frac{ (\sigma^2+G^2b)\alpha C_1}{\varepsilon}.
\end{align*}
Since this shows that we can multiply the result of $\alpha = \alpha$ by $\alpha^{-1}\gamma^{-n} > 0$,
we can immediately complete the proof from the above discussion.
\end{proof}

Table \ref{tab:rel} shows the relationship between the number of steps $K$ and batch size $b$ for each step size
\begin{table}[htbp]
\caption{Relationship between number of steps $K$ needed for $\varepsilon$-approximation and batch size $b$ ($\gamma \in (0, 1)$, $G, \alpha, C_1, C_2, \varepsilon > 0$, and $\sigma^2 \geq 0$). \label{tab:rel}}
\centering
\begin{tabular}{ccc}
\bottomrule
\hline
\multirow{2}{*}{Step size $\alpha_k$} & Relationship & Lower bound \\
 & between $K$ and $b$ & of $b$ \\ \hline
\multirow{2}{*}{$\alpha_k = \alpha$} & \multirow{3}{*}{$K = \dfrac{C_2b}{\varepsilon b - (\sigma^2+G^2b)\alpha C_1}$} & \multirow{3}{*}{$\dfrac{ (\sigma^2+G^2b)\alpha C_1}{\varepsilon}$} \\
 &  &  \\
(Constant) & & \\ \hline
 \multirow{2}{*}{$\alpha_k = 1/\sqrt{k+1}$} & \multirow{3}{*}{$K = \left(\dfrac{2C_1\sigma^2 + (2C_1G^2+C_2)b}{\varepsilon b}\right)^2$} & \multirow{3}{*}{--} \\
 &  &  \\
(Diminishing) & & \\ \hline
\multirow{2}{*}{$\alpha_k = \alpha\gamma^{k}$} & \multirow{3}{*}{$K = \dfrac{C_2b\alpha^{-1}\gamma^{-n}}{\varepsilon b - (\sigma^2+G^2b)\alpha C_1}$}  & \multirow{3}{*}{$\dfrac{ (\sigma^2+G^2b)\alpha C_1}{\varepsilon}$} \\
 &  &  \\
(Diminishing) & & \\ \hline
\toprule
\end{tabular}
\end{table}

\section{Numerical Experiments}
\label{sec:experiments}
The experiments were run on a MacBook Air (2020) laptop with a 1.8 GHz Intel Core i5 CPU, 8 GB 1600 MHz DDR3 memory, and the Monterey operating system (version 12.2).
The algorithms were written in Python 3.10.7 using the PyTorch 1.13.1 package and the Matplotlib 3.6.2 package.
The code is available at \url{https://github.com/iiduka-researches/rsgd-kylberg.git}.

\subsection{Geometry of Symmetric Positive Definite Manifold}
The set of $d \times d$ SPD matrices
\begin{align*}
\spd^d_{++} := \{P \in \real^{d \times d} &: P^\top = P, \forall x \in \real^d - \{0\}, \, x^\top Px > 0 \}
\end{align*}
endowed with the affine-invariant metric,
\begin{align*}
\ip{X_P}{Y_P}_P :=& \tr(X_P^\top P^{-1}Y_PP^{-1}),
\end{align*}
where $P \in \spd^d_{++}$ and $X_P,Y_P \in T_P\spd^d_{++}$, is a $d(d+1)/2$ dimensional Hadamard manifold
(i.e., sectional curvatures of $\spd^d_{++}$ are less than or equal to zero).
This Riemannian manifold is called an ``SPD manifold with an affine-invariant Riemannian metric" \cite{pennec2006riemannian, ferreira2006newton, sato2019riemannian, gao2020learning}.
Criscitiello and Boumal \cite{criscitiello2022accelerated} showed that the sectional curvatures of $\spd^d_{++}$ are at least $-1/2$.

The exponential map $\Exp_P:T_P\spd^d_{++}\rightarrow\spd^d_{++}$ at a point $P \in \spd^d_{++}$ was computed using
\begin{align*}
\Exp_P(X_P) = P^{\frac{1}{2}}\exp\left(P^{-\frac{1}{2}}X_PP^{-\frac{1}{2}}\right)P^{\frac{1}{2}},
\end{align*}
where $X_P \in T_P\spd^d_{++}$ \cite{ferreira2006newton}.

\subsection{Riemannian centroid problem on SPD manifold}
We considered the Riemannian centroid problem of a set of SPD matrices $\{A_i\}_{i=0}^n$,
which is frequently used for computer vision problems such as visual object categorization and pose categorization \cite{jayasumana2015kernel}.
The loss function can be expressed as
\begin{align*}
f(M) := \frac{1}{N}\sum_{i=0}^N\norm{\log\left(A_i^{-\frac{1}{2}}MA_i^{-\frac{1}{2}}\right)}_F^2,
\end{align*}
where $\norm{\cdot}_F$ is the Frobenius norm.

We followed preprocessing steps similar to ones used elsewhere \cite{harandi2014bregman} and used the Kylberg dataset \cite{kylberg2011kylberg},
which contains 28 texture classes of different natural and human-made surfaces. Each class has 160 unique samples imaged with and without rotation.
The original images were scaled to $128 \times 128$ pixels and covariance descriptors were generated from 1024 $4 \times 4$ non-overlapping pixel grids.
The feature vector at each pixel was represented as
\begin{align*}
x_{u, v} = \left[ I_{u,v}, \abs{\frac{\partial I}{\partial u}}, \abs{\frac{\partial I}{\partial v}}, \abs{\frac{\partial^2 I}{\partial u^2}}, \abs{\frac{\partial^2 I}{\partial v^2}} \right],
\end{align*}
where $I_{u,v}$ is the intensity value.

We evaluated Algorithm \ref{alg:RSGD} for several batch sizes by solving the Riemannian centroid problem on an SPD manifold on the Kylberg dataset.
We used three learning rates: $\alpha_k = \alpha$ (constant), $\alpha_k = 1 / \sqrt{k+1}$ (diminishing1), and $\alpha_k = \alpha\gamma^{p_k}$ (diminishing2).
We used $\alpha = 5 \times 10^{-4}$, $\gamma = 0.5$, and $n = 10$.
We defined $T$ as the number of steps until all the elements in the data set were used once.
Numerical experiments were performed for all batch sizes between $2^4$ and $2^9$.

\subsection{Numerical Results}
Figures \ref{fig:steps1} and \ref{fig:steps2} show the number of steps $K$ needed for $f(x_k) < \varepsilon$ initially decreased for Algorithm \ref{alg:RSGD} versus batch size.
They show the results for $\varepsilon = 1 / 2$ and $\varepsilon = 1 / 4$, respectively.
Supporting the results shown in Section \ref{sec:critical}, the number of steps $K$ is monotone decreasing and convex with respect to batch size $b$.
Figures \ref{fig:sfo1} and \ref{fig:sfo2} plot SFO complexity $Kb$ for the number of steps $K$ needed to satisfy $f(x_k) < \varepsilon$ versus batch size $b$.
They show the results for $\varepsilon = 1 / 2$ and $\varepsilon = 1 / 4$, respectively.
Supporting the results shown in Section \ref{sec:critical}, SFO complexity $Kb$ is convex with respect to $b$.

Figure \ref{fig:sfo1} shows that if $\varepsilon = 1/2$, the critical batch sizes for the constant, diminishing1, and diminishing2 learning rates are $247$, $205$, and $247$, respectively.
Figure \ref{fig:sfo2} shows that if $\varepsilon = 1/4$, the critical batch sizes are $2^8$, $233$ and $ 2^8$, respectively.
These results support Theorem \ref{thm:critical}, which implies that the constant and diminishing2 learning rates have the same critical batch size and that it decreases as $\varepsilon$ is increased.
As indicated by Theorem \ref{thm:critical}, the critical batch size of diminishing1 for $\varepsilon = 1/2$ is almost the same as that of diminishing1 for $\varepsilon = 1/4$.

\begin{figure}[htbp]
\centering
\includegraphics[width=0.8\columnwidth]{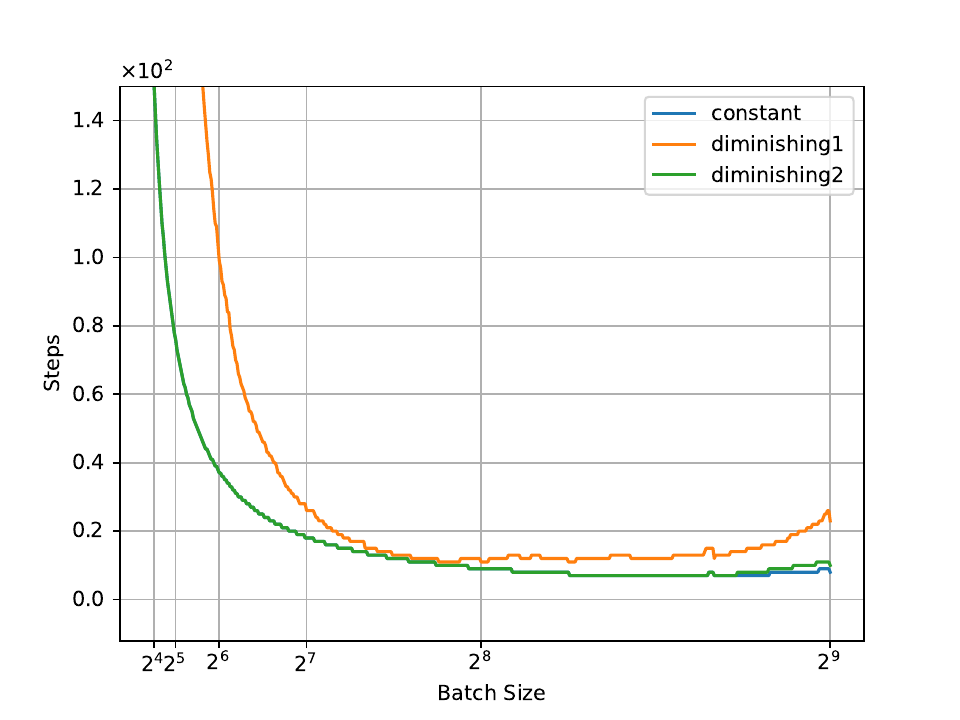}
\caption{Number of steps $K$ for Algorithm \ref{alg:RSGD} versus batch size $b$ when $\varepsilon = 1/2$.}
\label{fig:steps1}
\end{figure}

\begin{figure}[htbp]
\centering
\includegraphics[width=0.8\columnwidth]{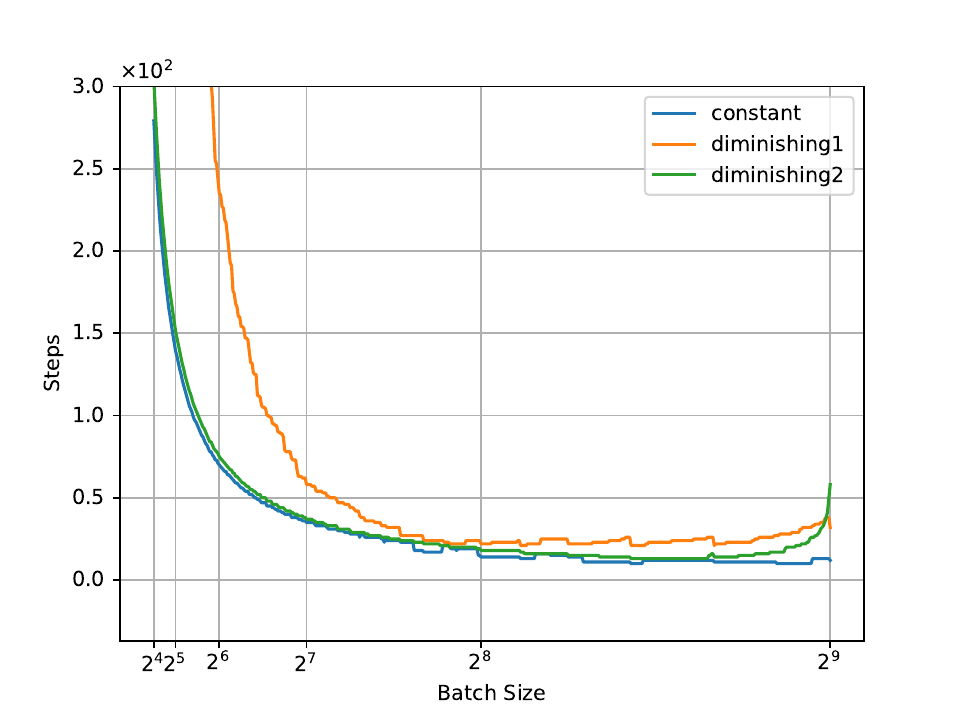}
\caption{Number of steps $K$ for Algorithm \ref{alg:RSGD} versus batch size $b$ when $\varepsilon = 1/4$.}
\label{fig:steps2}
\end{figure}

\begin{figure}[htbp]
\centering
\includegraphics[width=0.8\columnwidth]{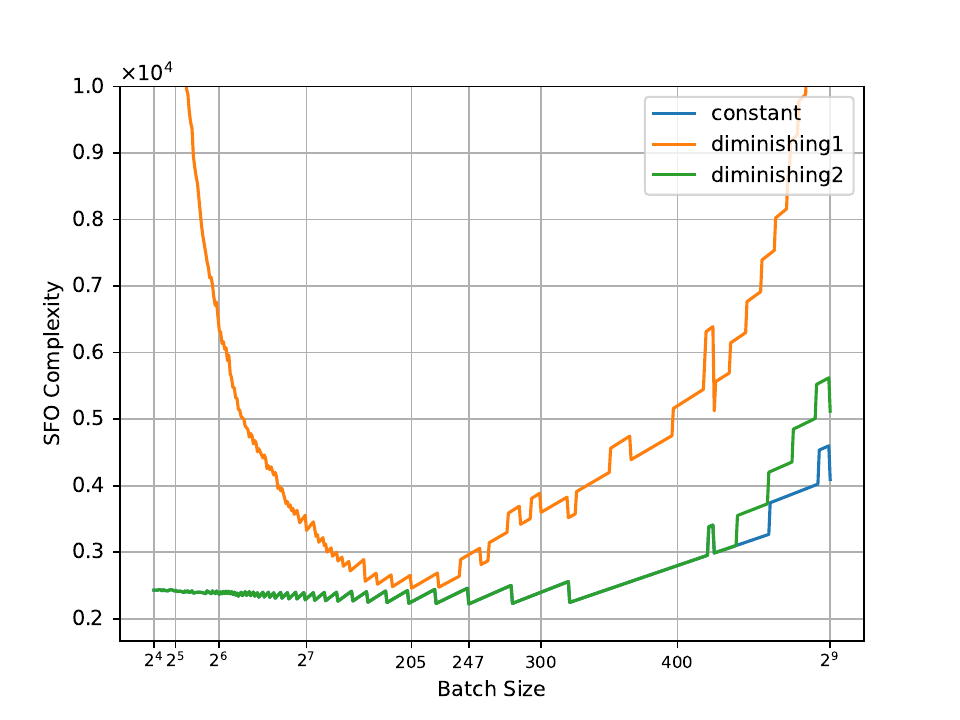}
\caption{SFO complexity $Kb$ for Algorithm \ref{alg:RSGD} versus batch size $b$ when $\varepsilon = 1/2$.}
\label{fig:sfo1}
\end{figure}

\begin{figure}[htbp]
\centering
\includegraphics[width=0.8\columnwidth]{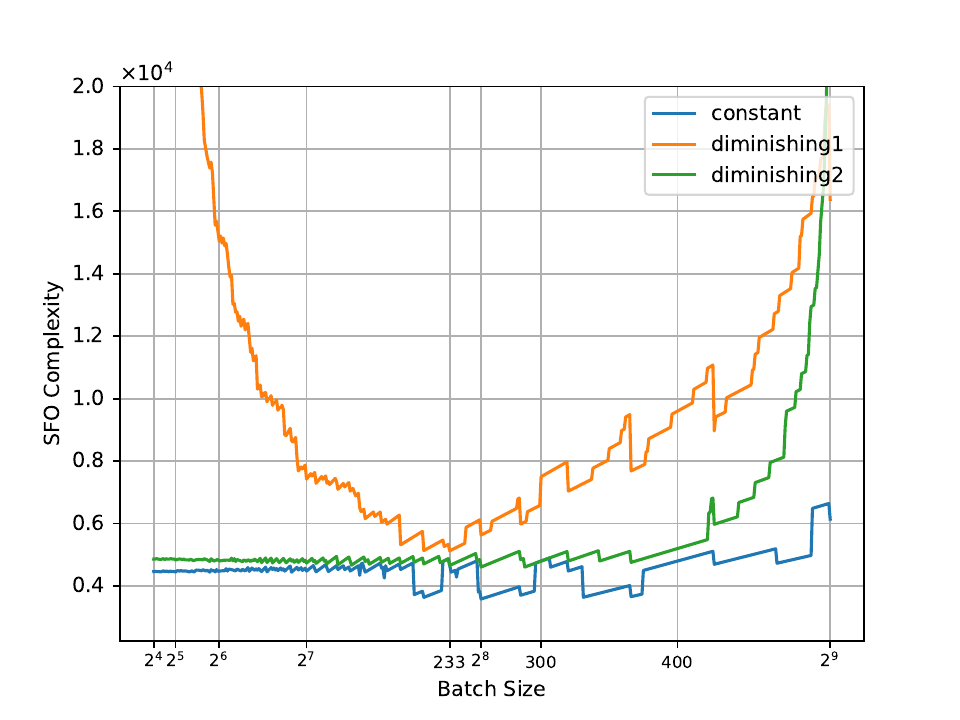}
\caption{SFO complexity $Kb$ for Algorithm \ref{alg:RSGD} versus batch size $b$ when $\varepsilon = 1/4$.}
\label{fig:sfo2}
\end{figure}

\section{Conclusion}
\label{sec:conclusion}
Our novel convergence analyses of Riemannian stochastic gradient descent on a Hadamard manifold, which incorporate the concept of mini-batch learning, overcome several problems in previous analyses.
We analyzed the relationship between batch size and the number of steps and demonstrated the existence of a critical batch size.
In practice, the number of steps for $\varepsilon$-approximation is monotone decreasing and convex with respect to batch size.
Moreover, stochastic first-order oracle complexity is convex with respect to batch size, and there exists a critical batch size that minimizes this complexity.
Numerical experiments in which we solved the Riemannian centroid problem on a symmetric positive definite manifold were performed using several batch sizes to verify the results of theoretical analysis.
With a constant step size, as $\varepsilon$ decreases, the critical batch size increases.
With a diminishing step size ($\alpha_k = \alpha \gamma^k$), the critical batch size matches that for the constant step size.
Therefore, the experiments give a numerical evidence of the results of theoretical analysis.



\bibliographystyle{abbrv}
\bibliography{biblio}

\end{document}